\documentclass[jsl]{asl}

\usepackage[latin1]{inputenc}
\usepackage{amssymb}
\title{{\bf  A hierarchy of tree-automatic  structures  } }



\author{Olivier Finkel}
\revauthor{Finkel, Olivier }

\address{{\it Equipe de Logique Math\'ematique  \\ Institut de Math\'ematiques de Jussieu}
 \\CNRS and  Universit\'e Paris 7,  France.}

\email{finkel@logique.jussieu.fr}

\author{Stevo Todor{\v{c}}evi{\'c}}
\revauthor{Todor{\v{c}}evi{\'c}, Stevo}

\twoaddress{{\it Equipe de Logique Math\'ematique  \\  Institut de Math\'ematiques de Jussieu}
 \\CNRS and  Universit\'e Paris 7,  France.}{ {\it  Department of Mathematics}
\\ University of Toronto, Toronto, Canada
M5S 2E4.}{2}

\email{stevo@logique.jussieu.fr}

\newtheorem{The}{Theorem}[section]
\newtheorem{Pro}[The]{Proposition}
\newtheorem{Deff}[The]{Definition}
\newtheorem{Lem}[The]{Lemma}
\newtheorem{Rem}[The]{Remark}

\newtheorem{Cor}[The]{Corollary}

\newtheorem{Remark}[The]{Remark}

\newcommand{\fa}{\forall}
\newcommand{\Ga}{\Gamma}

\newcommand{\Si}{\Sigma}
\newcommand{\Sis}{\Sigma^\star}
\newcommand{\Sio}{\Sigma^\omega}
\newcommand{\ra}{\rightarrow}
\newcommand{\hs}{\hspace{12mm}

\noi}
\newcommand{\hsn}{\hspace{12mm}

}

\newcommand{\ite}{\item}

\newcommand{\om}{\omega}
\newcommand{\nl}{\newline}
\newcommand{\noi}{\noindent}

\begin{document}

\begin{abstract}
\noi   We consider $\om^n$-automatic structures which are relational structures whose
domain and relations are accepted by automata reading ordinal words of length $\om^n$ for some integer $n\geq 1$.
We show that all these structures are $\om$-tree-automatic structures presentable  by Muller or Rabin tree automata.
We prove that  the  isomorphism relation for
$\om^2$-automatic (resp. $\om^n$-automatic for $n>2$) boolean algebras (respectively, partial orders, rings, commutative rings,
non commutative rings, non commutative groups) is not determined by the  axiomatic system {\rm ZFC}.
We infer from the proof of the above result that
the  isomorphism problem for  $\om^n$-automatic boolean algebras, $n\geq 2$,  (respectively,  rings, commutative rings,
non commutative rings, non commutative groups)
is neither a $\Si_2^1$-set nor a  $\Pi_2^1$-set.
We obtain that there exist infinitely many $\om^n$-automatic, hence also $\om$-tree-automatic,
atomless  boolean algebras $\mathcal{B}_n$, $n\geq 1$, which are pairwise isomorphic under
the continuum hypothesis {\rm  CH} and pairwise non isomorphic under an alternate axiom {\rm AT}, strengthening a result of \cite{Fin-Tod}.

\end{abstract}

\maketitle

\section{Introduction}

An automatic structure is a  relational structure whose domain and relations are recognizable by finite automata reading finite words.  Automatic structures
 have very nice decidability and definability properties and have been much studied in the last few years,
see \cite{BlumensathGraedel00,BlumensathGraedel04,KNRS,NiesBSL,RubinPhd,RubinBSL}.
They form a subclass of the class of (countable) recursive structures where ``recursive" is replaced by
``recognizable by finite automata".
 Blumensath considered in \cite{Blumensath99} more powerful kinds of automata. If we replace automata by tree automata (respectively, B\"uchi automata
reading infinite words, Muller or Rabin tree automata reading infinite labelled trees) then we get the notion of  tree-automatic (respectively,
$\om$-automatic, $\om$-tree-automatic) structures.
Notice that  an $\om$-automatic or $\om$-tree-automatic structure may have uncountable cardinality.
All these kinds of automatic structures have the two following fundamental properties.
$(1)$ The class
of automatic (respectively,   tree-automatic, $\om$-automatic, $\om$-tree-automatic) structures is closed under first-order interpretations.
$(2)$ The first-order theory of an automatic (respectively,   tree-automatic, $\om$-automatic, $\om$-tree-automatic) structure is decidable. 

 On the other hand, automata reading words of  ordinal length had been firstly considered by B\"uchi  in his investigation of the decidability of 
the monadic second order theory of a countable ordinal, see \cite{bs, Hemmer} and also \cite{Woj,Woj2,Bedon96,BedonCarton98,Bedon2}
 for further references on the subject. 
We investigate in this paper $\om^n$-automatic structures which are relational structures whose
domain and relations are accepted by automata reading ordinal words of length $\om^n$ for some integer $n\geq 1$.
All these structures are $\om$-tree-automatic structures presentable  by Muller or Rabin tree automata.

  A fundamental question about  classes of automatic structures is
the following: ``what is the complexity of the isomorphism problem for some class of automatic structures?"
 The isomorphism problem
for the class of automatic structures, or  even for the class of automatic graphs,
is $\Si_1^1$-complete,   \cite{KNRS}.  
On the other hand,   the isomorphism problem is decidable for automatic ordinals or for
automatic  boolean algebras, see  \cite{KNRS,RubinBSL}. Some more results about other classes of automatic structures may be found   in 
\cite{lics-KuskeLL10}: in particular, the  isomorphism problem for automatic linear orders is not arithmetical.   
Hjorth, Khoussainov, Montalb{\'a}n, and Nies proved
that the isomorphism problem for  $\om$-automatic structures is not a   $\Si_2^1$-set, \cite{HjorthKMN08}. More Recently, Kuske, 
 Liu,  and Lohrey proved in \cite{KuskeLL-CSL10} 
that the isomorphism problem for  $\om$-automatic structures (respectively, partial orders, trees of finite height) 
is not even analytical, i.e. is  not in any class $\Si_n^1$ 
where $n\geq 1$ is an integer. 
In \cite{Fin-Tod} we recently proved  that the  isomorphism relation for $\om$-tree-automatic structures
 (respectively,  $\om$-tree-automatic boolean algebras,   partial orders, rings, commutative rings,
non commutative rings, non commutative groups)  is not determined by the  axiomatic system {\rm ZFC}.  
This showed the importance of different axiomatic systems of Set Theory in the area of $\om$-tree-automatic structures. 

 We prove here that  the  isomorphism relation for
$\om^2$-automatic (resp. $\om^n$-automatic for $n>2$) boolean algebras (respectively, partial orders, rings, commutative rings,
non commutative rings, non commutative groups) is not determined by the  axiomatic system {\rm ZFC}.
We infer from the proof of the above result that
the  isomorphism problem for  $\om^n$-automatic boolean algebras, $n\geq 2$,  (respectively,  rings, commutative rings,
non commutative rings, non commutative groups)
is neither a $\Si_2^1$-set nor a  $\Pi_2^1$-set.
We obtain that there exist infinitely many $\om^n$-automatic, hence also $\om$-tree-automatic,
atomless  boolean algebras $\mathcal{B}_n$, $n\geq 1$, which are pairwise isomorphic under
the continuum hypothesis {\rm  CH} and pairwise non isomorphic under an alternate axiom {\rm AT} (for   ``almost trivial"). 
This way we  improve our   result of \cite{Fin-Tod}, where  we used the open coloring axiom {\rm  OCA} instead,  in two ways: 

\begin{itemize}
\ite[(1)]  by constructing infinitely many structures with independent isomorphism problem, instead of only two such structures, and 
\ite[(2)] by finding such structures which are much simpler than the $\om$-tree-automatic ones, because they are even $\om^n$-automatic. 
\end{itemize}
 
 The paper is organized as follows. In Section 2 we recall definitions and first properties of automata reading ordinal words and of tree automata. 
In Section 3 we define $\om^n$-automatic structures and $\om$-tree-automatic  structures and  we   prove simple properties of  
$\om^n$-automatic structures. We introduce in Section 4 some  particular   $\om^n$-automatic boolean algebras $\mathcal{B}_n$. 
We recall in Section 5 some results of  Set Theory and recall in particular the  Axiom {\rm AT} (`almost trivial')  
and some related notions. We investigate in Section 6 the isomorphism relation for $\om^n$-automatic structures and for $\om$-tree-automatic  structures. 
 Some concluding remarks are given in Section 7.

\section{Automata}

\subsection{$\om^n$-Automata}

\noi When $\Si$ is a finite alphabet, a {\it non-empty finite word} over $\Si$ is any
sequence $x=a_0.a_1\ldots a_k$, where $a_i\in\Sigma$
for $i=1,\ldots ,k$ , and  $k$ is an integer $\geq 0$. The {\it length}
 of $x$ is $k+1$.
 The {\it empty word} has no letter and is denoted by $\varepsilon$; its length is $0$.
 For $x=a_0.a_1\ldots a_k$, we write $x(i)=a_i$.
 $\Sis$  is the {\it set of finite words} (including the empty word) over $\Sigma$.

  We  assume  the reader to be familiar with the elementary theory of countable ordinals.
Let $\Si$ be a finite  alphabet, and  $\alpha$ be an ordinal; a word of length
 $\alpha$ (or  $\alpha$-word) over the alphabet $\Si$ is an  $\alpha$-sequence $(x(\beta))_{\beta < \alpha}$
(or sequence of length $\alpha$) of letters in $\Si$. The set of $\alpha$-words over the alphabet $\Si$ is denoted by $\Si^\alpha$.
The concatenation of an $\alpha$-word $x=(x(\beta))_{\beta < \alpha}$ and of a $\gamma$-word $y=(y(\beta))_{\beta < \gamma}$ is the 
$(\alpha + \gamma)$-word $z=(z(\beta))_{\beta < \alpha+ \gamma}$  such that $z(\beta)=x(\beta)$ for $\beta < \alpha$ and 
$z(\beta)=y(\beta')$ for $\alpha \leq \beta=\alpha + \beta' < \alpha+ \gamma$; it is denoted $z = x\cdot y$ or simply $z = xy$.

We  assume that the reader is familiar with the notion of B\"uchi automaton reading
infinite words over a finite alphabet which can be found for instance in \cite{Thomas90,Staiger97}.  Informally speaking an
$\om$-word $x$ over $\Si$  is accepted by a B\"uchi automaton $\mathcal{A}$ iff there is an infinite run of
$\mathcal{A}$ on $x$  enterring  infinitely often in  some final state of $\mathcal{A}$. The $\om$-language
$L(\mathcal{A})\subseteq \Sio$  accepted by the B\"uchi automaton $\mathcal{A}$ is
the set of  $\om$-words $x$ accepted by $\mathcal{A}$.  A Muller automaton is a finite automaton equipped with a set $\mathcal{F}$
of accepting sets of states.  An $\om$-word $x$ over $\Si$  is accepted by a Muller  automaton $\mathcal{A}$ iff there is an infinite run of
$\mathcal{A}$ on $x$  such that the set of states appearing   infinitely often during this run is an accepting set of states, i.e. belongs to  $\mathcal{F}$.
It is well known that an $\om$-language is accepted by a  B\"uchi automaton iff it is accepted by a Muller automaton.

 We shall define $\om^n$-automatic structures as relational structures presentable by automata reading words of length $\om^n$, for some
integer $n\geq 1$.
In order to read some words of transfinite length greater than $\om$, an   automaton  must have a  transition relation for successor steps defined as usual but also
a  transition relation for limit steps.  After the reading of a word whose  length is a limit  ordinal,
the state of the automaton will depend on the set of states
which cofinally appeared  during the run of the automaton.
These automata have been firstly considered by B\"uchi,  see   \cite{bs, Hemmer}.  We recall now their definition and behaviour.

\begin{Deff}[\cite{Woj,Woj2,Bedon96}]  An ordinal  Büchi automaton is a  sextuple
 $(\Sigma, Q,$ $q_0, \Delta, \gamma, F)$,  where
$\Si $ is a finite alphabet,
$Q$ is a finite set of states,
$q_0\in Q$ is the   initial state,
$\Delta \subseteq Q \times \Sigma \times Q$ is the transition relation for successor steps, and
$\gamma \subseteq P(Q) \times Q$ is the transition relation for limit steps.

\hs  A run of the ordinal   Büchi automaton
 $\mathcal{A}=(\Si, Q, q_0, \Delta , \gamma, F)$  reading a word  $\sigma$ of length
 $\alpha$, is an ($\alpha +1$)-sequence of states  $x$  defined by:
$x(0)=q_0$ and, for  $i<\alpha$, $(x(i), \sigma(i), x(i+1))\in \Delta$ and, for
 $i$  a limit ordinal, $(Inf(x,i),x(i)) \in \gamma$, where $Inf(x, i)$ is the set of states which cofinally appear during the reading
of the  $i$ first letters of  $\sigma$, i.e.
 $$Inf(x, i)=\{ q\in Q \mid  \fa \mu <i, \exists \nu<i  \mbox{ such that }
 \mu<\nu  \mbox{  and  } x(\nu)=q \}$$
\noi A run  $x$  of the automaton  $\mathcal{A}$ over the word
 $\sigma$
of length  $\alpha$ is called successful if $x(\alpha) \in F$. A word $\sigma$
of length  $\alpha$
is accepted  by  $\mathcal{A}$ if there exists a  successful run of $\mathcal{A}$ over $\sigma$. We denote
 $L_\alpha(\mathcal{A})$ the set of words of length  $\alpha$ which are accepted by  $\mathcal{A}$.
 An $\alpha$-language $L$ is a regular  $\alpha$-language if there exists an
 ordinal  Büchi automaton $\mathcal{A}$ such that  $L=L_\alpha(\mathcal{A})$.

\hs An ordinal  Büchi automaton $(\Sigma, Q, q_0, \Delta, \gamma, F)$ is said to be deterministic iff
$\Delta \subseteq Q \times \Sigma \times Q$ is in fact the graph of a
function from $Q \times \Sigma$ into $Q$ and $\gamma \subseteq P(Q) \times Q$ is the graph of a function from $P(Q)$ into $Q$. In that case there is at most
one run of the automaton over a given word  $\sigma$.
\end{Deff}

\begin{Rem} When we consider only finite words, the language accepted by an
ordinal  Büchi automaton
is a  rational language. If we consider only $\om$-words, the  $\omega$-languages acceped by
 ordinal  Büchi automata  are the  $\omega$-languages accepted
by Muller automata and then also by Büchi automata.
\end{Rem}

\begin{Deff}
An  $\om^n$-automaton is an ordinal  Büchi  automaton reading only words of length $\om^n$  for some integer $n\geq 1$.
\end{Deff}

 We can obtain  regular  $\om^n$-languages from regular $\om$-languages and regular $\om^{n-1}$-languages by the use of the  notion of substitution.
The following  result 
appeared in \cite{Hemmer} and has been  also proved in \cite{Finkel-loc01}.

\begin{Pro}\label{sub} Let $n\geq 2$ be an integer.
An $\om^n$-language $L \subseteq \Si^{\om^n}$  is  regular iff it is obtained from a regular
$\om$-language  $R \subseteq \Ga^\om$
by substituting in every $\om$-word $\sigma \in R$
a regular  $\om^{n-1}$-language $L_a  \subseteq \Si^\om$ to each letter $a\in \Ga$.
\end{Pro}

\hsn We now recall some fundamental  properties of regular $\om^n$-languages.

\begin{The}[B\"uchi-Siefkes, see \cite{bs, Hemmer,Bedon96}]  Let $n\geq 1$ be an integer.
One can effectively decide whether the  $\om^n$-language $L(\mathcal{A})$  accepted by a given $\om^n$-automaton    $\mathcal{A}$ is empty or not.
\end{The}

\begin{The}\label{reg}[see \cite{Bedon96}]  Let $n\geq 1$ be an integer.
The class of regular  $\om^n$-languages  is effectively closed under finite union, finite intersection, and complementation, i.e.
we can effectively construct, from two $\om^n$-automata    $\mathcal{A}$ and $\mathcal{B}$, some
$\om^n$-automata
$\mathcal{C}_1$, $\mathcal{C}_2$, and $\mathcal{C}_3$,  such that $L(\mathcal{C}_1)=L(\mathcal{A}) \cup L(\mathcal{B})$,
$L(\mathcal{C}_2)=L(\mathcal{A}) \cap L(\mathcal{B})$,  and $L(\mathcal{C}_3)$ is the complement of $L(\mathcal{A})$.
\end{The}

\hsn  We assume the reader to be familiar with basic notions of topology that
may be found in \cite{Moschovakis80,Kechris94,PerrinPin}.
The usual  Cantor topology on $\Sio$ is  the product topology obtained from the discrete topology on the finite set $\Si$, 
 for which {\it open subsets} of
$\Sio$ are in the form $W\cdot \Si^\om$, where $W\subseteq \Sis$.

Let  $n\geq 1$ be an integer.  Let $B: \om \ra \om^n$ be a recursive bijection. Then we have a  bijection $\phi$ from
$\Si^{\om^n}$ onto $\Si^{\om}$ defined by $\phi(x)(n)=x(B(n))$ for each integer $n\geq 0$. Then for each $\om^n$-language $L\subseteq \Si^{\om^n}$ we have the associated
$\om$-language $\phi(L)=\{ \phi(x) \mid x\in L \}$.
Consider now a {\it regular} $\om^n$-language $L\subseteq \Si^{\om^n}$. It is stated in \cite{DFR} that  $\phi(L)$ is Borel (in the class 
 ${\bf \Sigma}_{2n+1}^0$).

\subsection{Tree automata}
 We introduce now  languages of infinite binary trees whose nodes
are labelled in a finite alphabet $\Si$.

 A node of an infinite binary tree is represented by a finite  word over
the alphabet $\{l, r\}$ where $r$ means ``right" and $l$ means ``left". Then an
infinite binary tree whose nodes are labelled  in $\Si$ is identified with a function
$t: \{l, r\}^\star \ra \Si$. The set of  infinite binary trees labelled in $\Si$ will be
denoted $T_\Si^\om$.
A tree language is a subset of $T_\Si^\om$, for some alphabet $\Si$.
(Notice that we shall only consider in the sequel {\it infinite} trees so we shall
often use the term tree instead of  {\it infinite} tree).

 Let $t$ be a tree. A branch $B$ of $t$ is a subset of the set of nodes of $t$ which
is linearly ordered by the prefix relation $\sqsubseteq$ and which
is closed under this prefix relation,
i.e. if  $x$ and $y$ are nodes of $t$ such that $y\in B$ and $x \sqsubseteq y$ then $x\in B$.
\nl A branch $B$ of a tree is said to be maximal iff there is not any other branch of $t$
which strictly contains $B$.

    Let $t$ be an infinite binary tree in $T_\Si^\om$. If $B$ is a maximal branch of $t$,
then this branch is infinite. Let $(u_i)_{i\geq 0}$ be the enumeration of the nodes in $B$
which is strictly increasing for the prefix order.
  The infinite sequence of labels of the nodes of  such a maximal
branch $B$, i.e. $t(u_0)t(u_1) \ldots t(u_n) \ldots $  is called a path. It is an $\om$-word
over the alphabet $\Si$.

 For a tree $t \in  T_\Si^\om$ and $u\in \{l, r\}^\star$, we shall denote  $t_u : \{l, r\}^\star \ra \Si$ the subtree defined by
$t_u(v)=t(uv)$ for all $v\in \{l, r\}^\star$. It is in fact the subtree of $t$ which is rooted in $u$.

 We are now going to define tree automata and regular languages of infinite trees.

\begin{Deff} A (nondeterministic) tree automaton  is a quadruple $\mathcal{A}=(Q,\Si,\Delta, q_0)$, where $Q$
is a finite set of states, $\Sigma$ is a finite input alphabet, $q_0 \in Q$ is the initial state
and $\Delta \subseteq  Q \times   \Si   \times  Q \times   Q$ is the transition relation.
 A run of the tree automaton  $\mathcal{A}$ on an infinite binary tree $t\in T_\Si^\om$ is an infinite binary tree $\rho \in T_Q^\om$ such that:
\nl (a)  $\rho (\varepsilon)=q_0$  and  ~~(b) for each $u \in \{l, r\}^\star$,  $(\rho(u), t(u), \rho(ul), \rho(ur))  \in \Delta$.

\hs 
A Muller  (nondeterministic) tree automaton  is a  5-tuple $\mathcal{A}=(Q,\Si,\Delta, q_0, \mathcal{F})$, where $(Q,\Si,\Delta, q_0)$ is a
tree automaton and $\mathcal{F} \subseteq 2^Q$ is the collection of  designated state sets.
 A run $\rho$ of the  Muller  tree automaton $\mathcal{A}$ on an infinite binary tree $t\in T_\Si^\om$ is said to be accepting if
for each path $p$ of $\rho$, the set of   states appearing infinitely  often on this path is in $\mathcal{F}$.
 The tree language $L(\mathcal{A})$ accepted by the  Muller tree automaton $\mathcal{A}$ is the set of infinite binary trees $t\in T_\Si^\om$
such that there is (at least) one accepting run of $\mathcal{A}$ on $t$.
 A tree language $L \subseteq T_\Si^\om$  is regular iff there exists a  Muller automaton $\mathcal{A}$ such that $L=L(\mathcal{A})$.
\end{Deff}

 We now recall some fundamental closure properties of regular  tree languages.

\begin{The}[Rabin, see \cite{Rabin69,Thomas90,2001automata,PerrinPin}]
The class of regular  tree languages is effectively closed under finite union, finite intersection, and complementation, i.e.
we can effectively construct, from two  Muller  tree automata   $\mathcal{A}$ and $\mathcal{B}$, some
 Muller tree automata
$\mathcal{C}_1$, $\mathcal{C}_2$, and $\mathcal{C}_3$,  such that $L(\mathcal{C}_1)=L(\mathcal{A}) \cup L(\mathcal{B})$,
$L(\mathcal{C}_2)=L(\mathcal{A}) \cap L(\mathcal{B})$,  and $L(\mathcal{C}_3)$ is the complement of $L(\mathcal{A})$.
\end{The}

\section{Automatic structures}
\noi Notice that one can  consider a relation $R \subseteq \Si_1^{\om^n} \times \Si_2^{\om^n}  \times \ldots \times \Si_k^{\om^n}$,
where $\Si_1, \Si_2, \ldots \Si_k$, are  finite alphabets,
as an $\om^n$-language over the product alphabet $\Si_1 \times \Si_2  \times \ldots \times \Si_k$.  In a similar way, we can consider a relation
$R \subseteq T_{\Si_1}^\om \times T_{\Si_2}^\om  \times \ldots \times T_{\Si_k}^\om$,
as a tree language over the product alphabet $\Si_1 \times \Si_2 \times \ldots \times \Si_k$.

Let now $\mathcal{M}=(M, (R_i^M)_{1\leq i\leq k})$  be  a relational structure,
  where $M$ is the domain,  and for each $i\in [1, k]$ ~ $R_i^M$ is a relation
of finite arity $n_i$ on the domain $M$. The structure is said to be  $\om^n$-automatic (respectively,  $\om$-tree-automatic)
if there is a presentation of the structure
where the domain and the relations on the domain are accepted by $\om^n$-automata (respectively, by Muller tree automata), in the following sense.

\begin{Deff}[see \cite{Blumensath99}]
Let $\mathcal{M}=(M, (R_i^M)_{1\leq i\leq k})$ be a relational structure, where $k\geq 1$ is an integer,  and each relation $R_i$ is of finite arity $n_i$.
\nl An  $\om$-tree-automatic presentation of the structure $\mathcal{M}$  is formed by   a tuple of  Muller tree  automata
$(\mathcal{A}, \mathcal{A}_=,  (\mathcal{A}_i)_{1\leq i\leq k})$,  and a mapping $h$ from $L(\mathcal{A})$ onto $M$,  such that:
\begin{enumerate}
\ite The automaton $\mathcal{A}_=$ accepts
an equivalence relation $E_\equiv $  on $L(\mathcal{A})$,  and
\ite
For each $i \in [1, k]$, the automaton $\mathcal{A}_i$ accepts an $n_i$-ary relation $R'_i$ on
$L(\mathcal{A})$ such that $E_\equiv$ is compatible with $R'_i$, and
\ite   The mapping $h$ is an isomorphism from the quotient  structure \nl $( L(\mathcal{A}),  (R'_i)_{1 \leq i \leq k} ) / E_\equiv$ onto $\mathcal{M}$.
\end{enumerate}

\noi  The $\om$-tree-automatic presentation is said to be injective if the equivalence relation $E_\equiv $ is just the equality relation on
$L(\mathcal{A})$. In this case  $\mathcal{A}_=$ and $E_\equiv $  can be omitted and
$h$ is simply an isomorphism from $( L(\mathcal{A}),  (R'_i)_{1 \leq i \leq k} )$ onto $\mathcal{M}$.
\noi A relational structure is said to be (injectively) $\om$-tree-automatic if it has an (injective) $\om$-tree-automatic presentation.
\end{Deff}

 Notice that sometimes an  $\om$-tree-automatic presentation is only given by a tuple of Muller tree  automata
$(\mathcal{A}, \mathcal{A}_=,  (\mathcal{A}_i)_{1\leq i\leq k})$, i.e. {\it without the mapping} $h$. In that case we still get the
$\om$-tree-automatic structure $( L(\mathcal{A}),  (R'_i)_{1 \leq i \leq k})  / E_\equiv$ which is in fact  equal to  $\mathcal{M}$  up to isomorphism.

 We get the definition of $\om^n$-automatic (injective) presentation of a  structure and of  $\om^n$-automatic structure  by simply replacing 
Muller tree  automata by $\om^n$-automata in the above definition.

 Notice  that, due to the good decidability properties of  Muller tree  automata and of  $\om^n$-automata, we can decide whether a given  automaton
$\mathcal{A}_=$ accepts
an equivalence relation $E_\equiv $  on $L(\mathcal{A})$ and whether, for each $i \in [1, k]$,  the
automaton $\mathcal{A}_i$ accepts an $n_i$-ary relation $R'_i$ on
$L(\mathcal{A})$ such that $E_\equiv$ is compatible with $R'_i$.

We denote
$\om^n$-AUT the class of $\om^n$-automatic structures 
and $\om$-tree-AUT the class of $\om$-tree-automatic structures.

\hsn  We state now two important properties of automatic structures.

\begin{The}[see \cite{Blumensath99}]
The class of  $\om$-tree-automatic  (respectively,  $\om^n$-automa \newline -tic) structures is closed under first-order interpretations. In other words  if
$\mathcal{M}$ is an $\om$-tree-automatic  (respectively,  $\om^n$-automatic) structure and $\mathcal{M}'$ is a relational structure which is
first-order interpretable in the structure $\mathcal{M}$, then the structure  $\mathcal{M}'$ is also $\om$-tree-automatic  (respectively,  $\om^n$-automatic).
\end{The}

\begin{The}[see \cite{Hodgson,Blumensath99}]\label{dec}
The first-order theory of an $\om$-tree-automatic (respectively,  $\om^n$-automatic) structure is decidable.
\end{The}

 Notice that $\om$-tree-automatic (respectively,   $\om^n$-automatic) structures are always {\it relational} structures.
 However we can also consider structures equipped with
functional operations like groups, by replacing as usually a $p$-ary function by its graph which is a $(p+1)$-ary relation.
This will always be the case in the sequel where all structures are viewed as relational  structures.

  Some examples of  $\om$-automatic structures can be found in
\cite{RubinPhd,NiesBSL,KNRS,KhoussainovR03,BlumensathGraedel04,KuskeLohrey,HjorthKMN08,KuskeLL-CSL10}.

 A first one is  the boolean algebra $\mathcal{P}(\om)$ of subsets of $\om$. 

 The additive group $(\mathbb{R}, +)$ is $\om$-automatic, as is the product $(\mathbb{R}, +) \times (\mathbb{R}, +)$.

 Assume that a finite alphabet $\Si$ is linearly ordered.
Then the  set   $(\Si^{\om^n}, \leq_{lex})$ of
 $\om^n$-words over the alphabet $\Si$, equipped with the lexicographic ordering, is  $\om^n$-automatic.

 Is is easy to see that every (injectively) $\om$-automatic structure is also
(injectively) $\om$-tree-automatic. Indeed a Muller tree automaton can easily simulate a B\"uchi automaton on the leftmost branch of an infinite  tree.

The inclusions $\om^n$-AUT $ \subseteq$ $\om^{n+1}$-AUT,  $n\geq 1$, are  straightforward  to prove.

\begin{Pro}
For each integer $n\geq 1$, ~~~ $\om^n$-AUT $\subseteq  \om$-tree-AUT.
\end{Pro}
  
\begin{proof} 
 We are first going to associate a tree $t^x$ to each $\om^n$-word $x$ in such a way that if $L \subseteq \Si^{\om^n}$ is a regular $\om^n$-language
then the tree language $\{ t^x \in T_\Si^\om \mid x\in L \}$ will be also  regular.
We make this by induction on the integer $n$. Let then $\Si$ be a finite alphabet and $a\in \Si$ be
a distinguished letter in $\Si$. We begin with the case $n=1$. If   $x\in \Sio$ is an $\om$-word over the alphabet $\Si$ then
$t^x$ is the tree in $T_\Si^\om$ such that $t^x(l^k)=x(k)$ for every integer $k\geq 0$ and $t^x(u)=a$ for every word $u\in \{l, r\}^\star$ such that
$u \notin \{ l^k \mid k\geq 0 \}$. It is clear that if $L \subseteq \Sio$ is a regular $\om$-language  then the tree language $\{ t^x \in T_\Si^\om \mid x\in L \}$
is a regular set of trees.
 Assume now that we have associated, for a given integer $n\geq 1$,  a tree $t^x$ to each $\om^n$-word $x \in \Si^{\om^n}$ in such a way that if
$L \subseteq \Si^{\om^n}$ is a regular $\om^n$-language
then the tree language $\{ t^x \in T_\Si^\om \mid x\in L \}$ is also  regular.  Consider now an $\om^{n+1}$-word $x$ over $\Si$. It can be divided into
$\om$ subwords $x_j$, $0\leq j <\om$, of length $\om^n$.  By induction hypothesis to each $\om^n$-word $x_j$ is associated a tree $t^{x_j} \in T_\Si^\om$.
Recall that we denote by $t_u$ the subtree of $t$ which is rooted in $u$.
We can now associate to the $\om^{n+1}$-word $x$ the tree $t^x$ which is defined by: $t^x_{l^k\cdot r }=t^{x_k}$ for every integer $k\geq 0$, and
$t^x(l^k)=a$ for every integer $k\geq 0$.
Let then now $L \subseteq \Si^{\om^{n+1}}$ be a regular $\om^{n+1}$-language. By  Proposition \ref{sub} the language $L$ is obtained from a regular
$\om$-language  $R \subseteq \Ga^\om$  by substituting in every $\om$-word $\sigma \in R$
a regular  $\om^{n}$-language $L_b \subseteq \Si^\om$ to each letter $b\in \Ga$.  By induction hypothesis for each letter $b\in \Ga$  there is a tree
automaton $\mathcal{A}_{b}$ such that  $L(\mathcal{A}_b)=\{ t^x \in T_\Si^\om \mid x\in L_b \}$. This implies easily that one can construct, from
a B\"uchi automaton accepting the regular $\om$-language  $R$ and from the tree automata  $\mathcal{A}_{b}$, $b\in \Ga$, 
another tree automaton $\mathcal{A}$ such that
$L(\mathcal{A})=\{ t^x \in T_\Si^\om \mid x\in L \}$.

 The  inclusion $\om^n$-AUT $\subseteq  \om$-tree-AUT holds because any element of the domain of an
$\om^n$-automatic structure, represented by an $\om^n$-word $x$, can also be  represented by a tree $t^x$. The relations
of the structure are then also presentable by tree automata. 
 \end{proof}

Notice that the strictness of the inclusion $\bigcup_{n\geq1}\om^n\mbox{-AUT} \subsetneq \om\mbox{-tree-AUT}$ follows easily from the 
existence of an $\om$-tree-automatic structrure without Borel presentation, proved in   \cite{HjorthKMN08}, and the fact that every 
$\om^n$-automatic structrure has a Borel presentation (see the end of Section 2.1). 

On the other hand we can easily see that the inclusion   $\om$-AUT $\subsetneq  \om^2$-AUT is strict by considering ordinals. 
Firstly, Kuske  recently  proved  in  \cite{Kuske10} that the $\om$-automatic ordinals are  the ordinals smaller than $\om^\om$. 
Secondly,  it is easy to see that the ordinal $\om^\om$ is $\om^2$-automatic. 
The ordinal $\om^\om$ is the order-type of finite sequences of integers ordered by
(1) increasing length of sequences and (2) lexicographical order for sequences of integers of the same length $n$.
A finite sequence of integers $x=(n_1, n_2, \ldots , n_p)$ can be represented by the following $\om^2$-word $\alpha_x$ over the alphabet $\{a, b\}$:
$$\alpha_x = (a^{n_1+1} \cdot b^\om) \cdot (a^{n_2+1} \cdot b^\om) \cdots  (a^{n_p+1} \cdot b^\om) \cdot (b^\om) \cdot (b^\om) \cdots $$
\noi it is then easy to see that there is an $\om^2$-automaton accepting exactly the $\om^2$-words of the form $\alpha_x$ for a finite sequence of integers $x$.
Moreover there is an  $\om^2$-automaton recognizing the pairs $(\alpha_x, \alpha_{x'})$ such that $x<x'$.

\section{Some $\om^n$-automatic boolean algebras}

\noi We have seen that the boolean algebra $\mathcal{P}(\om)$ of subsets of $\om$ is $\om$-automatic.
Another  known example of $\om$-automatic  boolean algebra is the  boolean algebra
$\mathcal{P}(\om)/{\rm Fin}$ of subsets of $\om$ modulo {\it finite sets}.
The set ${\rm Fin}$ of finite subsets of $\om$ is  an ideal of  $\mathcal{P}(\om)$, i.e.  a subset of the powerset
of $\om$ such that:
\begin{enumerate}
\ite $\emptyset \in {\rm Fin}$ and $\om \notin {\rm Fin}$.
\ite For all $B, B' \in {\rm Fin}$, it holds that $B \cup B' \in {\rm Fin}$.
\ite For all $B, B' \in \mathcal{P}(\om)$, if $B \subseteq B'$ and $B' \in {\rm Fin}$ then $B \in {\rm Fin}$.
\end{enumerate}

For any two subsets $A$ and $B$ of $\om$ we denote $A \Delta B$ their symmetric difference.
Then the relation $\approx$ defined by:
``$A \approx B$ iff the symmetric difference $A \Delta B$ is finite" is an equivalence relation on $\mathcal{P}(\om)$.
The quotient
$\mathcal{P}(\om) / \approx $ denoted  $\mathcal{P}(\om) / {\rm Fin}$ is a boolean algebra.
It is easy to see that this boolean algebra is $\om$-automatic, see for example \cite{KuskeLohrey,HjorthKMN08,Fin-Tod}.

 More generally we  now consider the boolean algebras  $\mathcal{P}(\om^n)/I_{\om^n}$ for integers $n\geq 1$.
We first give the definition of the sets  $I_{\om^n}\subseteq \mathcal{P}(\om^n)$.  For $P \subseteq \om^n$ we denote
$o.t.(P)$ the order type of $(P, <)$ as a suborder of  the order $(\om^n, <)$.  The set $I_{\om^n}$ is defined by:
$$I_{\om^n}=\{ P \subseteq \om^n \mid o.t.(P) < \om^n \}.$$

  For each integer $n\geq 1$ the set $I_{\om^n}$ is an  ideal of  $\mathcal{P}(\om^n)$.

For any two subsets $A$ and $B$ of $\om^n$ we denote $A \Delta B$ their symmetric difference.
Then the relation $\approx_n$ defined by:
``$A \approx_n B$ iff the symmetric difference $A \Delta B$ is in $I_{\om^n}$" is an equivalence relation on $\mathcal{P}(\om^n)$.
The quotient
$\mathcal{P}(\om^n) / \approx_n $, also denoted  $\mathcal{P}(\om^n)/I_{\om^n}$,  is a boolean algebra.

We are going to show that this boolean algebra $\mathcal{P}(\om^n)/I_{\om^n}$ is $\om^n$-automatic.

 We first notice that each  set $P \subseteq \om^n$ can be represented by an $\om^n$-word $x_P$  over the alphabet $\{0, 1\}$ by setting
$x_P(\alpha)=1$ if and only if $\alpha \in P$ for every ordinal $\alpha < \om^n$. Let then

$$L_n = \{ x_P \in \{0, 1\}^{\om^n} \mid  P \in I_{\om^n} \}.$$

\begin{The}\label{ln}  Let $n\geq 1$ be an integer. Then the set $L_n \subseteq \{0, 1\}^{\om^n}$ is a regular  $\om^n$-language.
\end{The}

\begin{proof}  We reason by induction on the integer $n$.  Firstly it is easy to see that $L_1$ is the set of $\om$-words over the alphabet
$\{0, 1\}$ having only finitely many letters $1$. It is a well known example of a regular $\om$-language, see \cite{PerrinPin,Thomas90}.
Notice that its complement $L_1^-=  \{0, 1\}^\om  \setminus L_1$        is then also regular since the class of regular $\om$-languages is closed
under complementation.

 We now assume that we have proved that for each integer $k \leq n$ the set  $L_k \subseteq \{0, 1\}^{\om^k}$ is a regular  $\om^k$-language.
In particular the language $L_n$ is a regular  $\om^n$-language. Moreover its complement $L_n^-=  \{0, 1\}^{\om^n}  \setminus L_n$
  is then also regular since the class of regular $\om^n$-languages is closed
under complementation.

Consider now a set $P \subseteq \om^{n+1}$. It is easy to see that $P$ belongs to $I_{\om^{n+1}}$ if and only if there are only finitely many
integers $k\geq 0$ such that $P \cap [ \om^n . k; \om^n .( k+1)[$  has order type $\om^n$.

Thus  the  $\om^{n+1}$-language  $L_{n+1}$  is obtained from the regular $\om$-language $L_1=\{0, 1\}^\star \cdot 0^\om$  by substituting
the $\om^n$-language $L_n^-$ to the letter $1$ and the $\om^n$-language $L_n$ to the letter $0$.
We can conclude, using Proposition \ref{sub}, that the $\om^{n+1}$-language  $L_{n+1}$ is regular.
\end{proof}

\hsn We can now state the following result.

\begin{The}  For every integer $n\geq 1$ the boolean algebra $\mathcal{P}(\om^n)/I_{\om^n}$ is $\om^n$-automatic.
\end{The}

\begin{proof} Let $n\geq 1$ be an integer.
We denote $[A]_{I_{\om^n}}$, or simply  $[A]$ when there is no confusion from the context, 
the equivalence class of a set  $A \subseteq \om^n$ for the equivalence relation $\approx_n$.
Let $\Si=\{0, 1\}$ and $L(\mathcal{A})= \Si^{\om^n}$ and for any $x\in \Si^{\om^n}$,
$h(x)=[\{ \alpha < \om^n \mid x(\alpha)=1\}]$.
Then it follows easily
from the preceding Theorem \ref{ln} that $\{ (u, v) \in (\Si^{\om^n})^2  \mid h(u)=h(v)  \}$ is accepted by an $\om^n$-automaton.

 The operations    $\cap, \cup, \neg$,
of intersection, union, and complementation,   on  $\mathcal{P}(\om^n)/I_{\om^n}$ are defined by:
  $[B]   \cap [B'] = [B  \cap  B'] $,    $[B]   \cup [B'] = [B  \cup  B'] $,  and
 $\neg [B] = [\neg B] $, see \cite{Jech}.

 Thus  the operations  of intersection, union, (respectively, complementation),
considered as ternary relations (respectively, binary relation)  are also
given by regular $\om^n$-languages.
On the other hand,   ${\bf 0}=[\emptyset]$ is the equivalence class of the empty set and ${\bf 1}=[\om^n]$ is the class of $\om^n$.

 This proves that the  structure $(\mathcal{P}(\om^n)/I_{\om^n}, \cap, \cup, \neg, {\bf 0},  {\bf 1})$
is $\om^n$-automatic.
\end{proof}

\hsn Notice that,  as in the above proof, we can see  that the relation
$\{ (u, v) \in (\Si^{\om^n})^2   \mid  h(u) \subseteq_n h(v)  \}$ is a regular $\om^n$-language because the ``almost inclusion" relation $\subseteq_n$
is defined by:  $h(u) \subseteq_n h(v)$ iff  $\{ \alpha < \om^n \mid  u(\alpha) > v(\alpha)  \} \in I_{\om^n}$. Thus we can also state the following result.

\begin{The}  For each integer $n\geq 1$ the structure $(\mathcal{P}(\om^n)/I_{\om^n}, \subseteq_n)$  is $\om^n$-automatic.

\end{The}

 From now on we shall  denote  $\mathcal{B}_n = (\mathcal{P}(\om^n)/I_{\om^n}, \cap, \cup, \neg, {\bf 0},  {\bf 1})$.  The boolean algebra
$\mathcal{B}_n$ is $\om^n$-automatic hence also  $\om$-tree-automatic.

\hsn Recall now the definition of an atomless boolean algebra.

\begin{Deff}
Let $\mathcal{B} = (B, \cap, \cup, \neg, {\bf 0},  {\bf 1})$ be a boolean algebra and $\subseteq$ be the inclusion relation on $B$ defined by
$x \subseteq y$ iff $x \cap y = x$ for all $x, y \in B$. Then the boolean algebra $\mathcal{B}$ is said to be an atomless boolean algebra iff for every
$x\in B$ such that $x\neq {\bf 0}$ there exists an element  $z\in B$ such that ${\bf 0} \subset z \subset x$.
\end{Deff}

 We can now recall  the following known  result.

\begin{Pro}\label{atomless}
 For each integer $n\geq 1$ the boolean algebra $\mathcal{B}_n$ is an  atomless boolean algebra.
\end{Pro}

\begin{proof}   Let $n\geq 1$ be an integer. Consider  the  boolean algebra $\mathcal{B}_n = (\mathcal{P}(\om^n)/I_{\om^n}, $ $\cap, \cup, \neg, {\bf 0},  {\bf 1})$.
Let $A \subseteq \om^n$ be  such that the equivalence class $[A]$ is different from
the element ${\bf 0}$ in $\mathcal{B}_n$.
Then the set $A$ has order type $\om^n$ and it can be splitted in two
 sets $A_1$ and $A_2$ such that $A=A_1 \cup A_2$ and both $A_1$ and $A_2$ have still order type $\om^n$.
The element $[A_1]$ is different from the element ${\bf 0}$ in $\mathcal{B}_n$ because $A_1$ has order type $\om^n$, and $[A_1] \subset [A]$ because
$A - A_1=A_2 $ has order type $\om^n$. Thus the following strict inclusions hold in   $\mathcal{B}_n$:
${\bf 0} \subset [A_1]  \subset [A]$.
 This proves that the boolean algebra $\mathcal{B}_n$ is atomless.
\end{proof}

The following result is also well known but we give a proof for completeness.

\begin{Pro}\label{tal} For each integer $n\geq 1$ the boolean algebra $\mathcal{B}_n$ has  the cardinality $2^{\aleph_0}$ of the continuum.
\end{Pro}

\begin{proof}
By recursion on $n\geq 1$ we define a partition of  $\omega^n$ into a sequence $F^n_k$ $(k<\omega)$
of nonempty finite sets such that for every infinite $X\subseteq \omega,$ the union $\Phi^n(X)=\bigcup_{k\in X} F^n_k$ does not belong to the ideal $I_{\omega^n}.$
For $n=1,$ let $F^1_k=\{k\}.$ For $n>1,$ set
$F^n_k=\bigcup_{\ell <k} F^{n-1}_k(\ell), $
where for each $\ell<\omega $ we have fixed a decomposition $F^{n-1}_k(\ell)$ $(k<\omega)$ of the interval $[\ell\omega^{n-1}, (\ell + 1)\omega^{n-1})$  of ordinals
into a sequence of finite nonempty pairwise disjoint sets with the property that $\bigcup_{k\in X}F^{n-1}_k(\ell)$  has order type $\om^{n-1}$ for every
infinite $X\subseteq \om.$

Clearly $\Phi^n: \mathcal{P}(\omega) \rightarrow \mathcal{P}(\omega^n)$ is a complete Boolean 
algebra embedding which induces also an embedding of $\mathcal{P}(\omega)/{\rm Fin}$ into 
$\mathcal{P}(\om^n)/I_{\om^n}.$ Since $\mathcal{P}(\omega)/{\rm Fin}$ has cardinality continuum the conclusion follows.
\end{proof}

In the sequel we shall often identify   the powerset $\mathcal{P}(A)$ of a countable set $A$ with the Cantor space $2^\om=\{0, 1\}^\om$. Then 
 $\mathcal{P}(A)$ can be  equipped with the 
standard metric topology obained from this  identification, and   
 the  topological notions like open, closed, ${\bf \Si}^0_2$, Borel,  analytic, can be applied 
to families of subsets of $A$.

\begin{Remark}
{\rm In fact a similar result holds for an arbitrary ${\bf \Sigma}^1_1$-ideal of subsets of $\om.$ 
More precisely, by a well-known result of Talagrand ( \cite{Tal} ; Th\'{e}or\`{e}me 21), for every proper ${\bf \Sigma}^1_1$-ideal $I$ 
of subsets of $\om$ there is
a partition of  $\omega$ into a sequence $F_n$ $(n<\omega)$
of nonempty finite sets such that for every infinite $X\subseteq \omega,$ the union $\bigcup_{n\in X} F_n$ does not belong to the ideal $I.$
Therefore, as above, there is an embedding of $\mathcal{P}(\omega)/{\rm Fin}$ into $\mathcal{P}(\om^n)/I.$ }
\end{Remark}

\begin{Remark}\label{rem2}
{\rm Recall that two subsets $X$ and $Y$ are said to be \emph{almost disjoint} if their intersection is finite. Recall that while a countable index set
does not admit an uncountable family of pairwise disjoint subsets it does admit an uncountable family of subsets that are pairwise almost disjoint. 
So fix an uncountable family $\mathcal{F}$ of pairwise almost disjoint infinite subsets of $\om.$ For $n\geq 1$ fix a sequence $F^n_k$ $(k<\omega)$
of nonempty finite subsets of $\om^n$ such that for every infinite $X\subset \omega,$ the union $\Phi^n(X)=\bigcup_{k\in X} F^n_k$ 
does not belong to the ideal $I_{\omega^n}$ (see the proof of Proposition \ref{tal}). For $X\in \mathcal{F}$ let $A_X=\bigcup_{k\in X} F^n_k.$ 
Then $A_X$ $(X\in \mathcal{F})$ is an uncountable
family of infinite subsets of $\om^n$ which is also almost disjoint (i.e., $A_X\cap A_Y \in {\rm Fin}$ for $X\neq Y$ in 
$\mathcal{F}$) but it has the additional property that
$A_X\not\in I_{\om^n}$ for all $X\in\mathcal{F}.$}
\end{Remark}

\section{Axioms of set theory}

\noi We now recall some basic notions of set theory
which will be useful in the sequel, and which are exposed in any  textbook on set theory, like \cite{Jech}.

 The usual axiomatic system {\rm ZFC} is
Zermelo-Fraenkel system {\rm ZF}   plus the axiom of choice {\rm AC}.
 A model ({\bf V}, $\in)$ of  the axiomatic system {\rm ZFC} is a collection  {\bf V} of sets,  equipped with
the membership relation $\in$, where ``$x \in y$" means that the set $x$ is an element of the set $y$, which satisfies the axioms of  {\rm ZFC}.
We shall often say `` the model {\bf V}"
instead of  ``the model  ({\bf V}, $\in)$".

 The axioms of {\rm ZFC} express some  natural facts that we consider to hold in the universe of sets. 

 The infinite cardinals are usually denoted by
$\aleph_0, \aleph_1, \aleph_2, \ldots , \aleph_\alpha, \ldots$

  We recall that Cantor's   Continuum
Hypothesis         ${\rm CH}$ states that the cardinality of the continuum $2^{\aleph_0}$ is equal
to the first uncountable cardinal $\aleph_1$.
 G\"odel and Cohen have proved that the continuum hypothesis {\rm CH} is independent from the axiomatic system {\rm ZFC}. This means that, assuming 
 {\rm ZFC} is consistent, there are some models of {\rm ZFC + CH} and also some models of {\rm ZFC + $\neg$ CH}, 
where {\rm $\neg$ CH} denotes the negation of the continuum hypothesis, \cite{Jech}.

  If  {\bf V} is  a model of {\rm ZF} and ${\bf L}$ is  the class of  {\it constructible sets} of   {\bf V}, then the class  ${\bf L}$     forms a model of
{\rm ZFC + CH}.

 Recall also that  ${\rm OCA}$ denotes the \emph{Open Coloring Axiom} (or \emph{Todorcevic's axiom} as it is called in the more recent literature; see, for example, \cite{Fa3}),  a natural
alternative to ${\rm CH}$ that has been first considered by
the second author in \cite{Todorcevic89}. 
It is known that if the theory {\rm ZFC} is
consistent, then so are the theories {\rm (ZFC + CH)} and {\rm
(ZFC + OCA)}, see \cite[pages 176 and  577]{Jech}.
 In particular, if {\bf V} is  a model of ${\rm (ZFC + OCA)}$ and if  ${\bf L}$ is  the class of  {\it constructible sets} of   {\bf V},
then the class  ${\bf L}$     forms a model of
${\rm (ZFC + CH)}$.

\hsn
The axiom {\rm OCA} was used in our previous paper \cite{Fin-Tod} on tree-automatic structures but here we shall use another related axiom
first considered by Just \cite{Ju1}. To introduce this axiom we need some definitions. 

Let $A$ and $B$ be two infinite countable sets,  a function $H: \mathcal{P}(A)\rightarrow\mathcal{P}(B)$ and an ideal $I$ of $\mathcal{P}(B)$ 
containing all finite subsets of $B$ but not the whole set $B$. 
Then the function $H$ is said to {\it preserve intersections modulo} $I$
whenever
\begin{enumerate}
\item[(i)] $H(X)\vartriangle H(Y)\in I$ for every $X,Y\subseteq A$ such that $X\vartriangle Y\in {\rm Fin},$ and
\item[(ii)] $H(X\cap Y)\vartriangle (H(X)\cap H(Y))\in I$ for every $X,Y\subseteq A$
\end{enumerate}

 Recall that one can identify the powerset  $\mathcal{P}(B)$, where $B$ is a countable set, with the set $2^{B}$ equipped with the Cantor topology which 
is the product topology of the discrete topology on $B$. Thus  one can also use notions like open, closed,  Borel, analytic, for ideals of $\mathcal{P}(B)$, 
where $n \geq 1$ is an integer.

\hsn Then Just's axiom ${\rm AT}$  (where the shorthand stands for `Almost Trivial') states that for every ${\bf \Sigma^1_1}$-ideal $I$ of subsets of $\om$, 
for every $H: \mathcal{P}(\om)\rightarrow\mathcal{P}(\om)$ which preserves intersections modulo $I$, 
and for every uncountable family $\mathcal{A}$ of pairwise almost disjoint infinite subsets of $\om$ there exist $A\in \mathcal{A}$ 
and a finite decomposition
$A=\bigcup_{i<n}A_i$ such that for every $i<n$ there is a \emph{continuous function}\footnote{Continuity here is interpreted 
when we make the standard identification
of $\mathcal{P}(A_i)$ and $\mathcal{P}(\om)$ with the Cantor cubes $2^{A_i}$ and 
$2^\om,$ respectively.} $F_i: \mathcal{P}(A_i)\rightarrow\mathcal{P}(\om)$ such that $F_i(X)\vartriangle H(X)\in I$
for every $X\subseteq A_i.$

\hsn  The axiom ${\rm AT}$ implies also the following form which will be used in the sequel:  
 For every ${\bf \Sigma^1_1}$-ideal $I$ of subsets of $\om^n$, 
for every $H: \mathcal{P}(\om^m)\rightarrow\mathcal{P}(\om^n)$ which preserves intersections modulo $I$, 
and for every uncountable family $\mathcal{A}$ of pairwise almost disjoint infinite subsets of $\om^m$ there exist $A\in \mathcal{A}$ 
and a finite decomposition
$A=\bigcup_{i<k}A_i$ such that for every $i<k$ there is a \emph{continuous function}
 $F_i: \mathcal{P}(A_i)\rightarrow\mathcal{P}(\om^n)$ such that $F_i(X)\vartriangle H(X)\in I$
for every $X\subseteq A_i.$

\hsn In \cite{Ju1}, Just showed that every model {\bf V} of {\rm ZFC} admits a forcing extension satisfying {\rm ZFC + AT}.
In another paper (\cite{Ju2}) he showed that {\rm OCA} implies many instances of {\rm AT}.
In particular, it is shown in \cite{Ju2} that {\rm OCA} implies {\rm AT} restricted to the class of all ${\bf \Sigma^0_2}$-ideals of subsets
of $\om.$ This was later extended by Farah \cite{Fa} to a  larger class of ${\bf \Sigma^1_1}$-ideals of subsets of $\om$, however it is still not known if {\rm OCA} implies the full {\rm AT.}
The motivation behind the axiom {\rm AT} came from the theory of quotient Boolean algebras of the form $\mathcal{P}(\om)/I$
where $I$ is a proper (i.e., $\om\not\in I$) ideal on $\om$ which we always assume to include the ideal ${\rm Fin}$ of all finite subsets
of $\om.$ Let $\pi_I: \mathcal{P}(\om)\rightarrow \mathcal{P}(\om)/I$ denotes the natural quotient map, i.e.
$\pi_I(X)=\pi_I(Y)$ whenever $X\vartriangle Y\in I.$ A homomorphism $$\Phi: \mathcal{P}(\om)/I\rightarrow \mathcal{P}(\om)/J$$
between two such quotient Boolean algebras is usually given by its \emph{lifting} $H: \mathcal{P}(\om)\rightarrow \mathcal{P}(\om)$ i.e.,
a map for which the following diagram

$$\begin{array}{ccccc}
   \mathcal{P}(\omega)        & \stackrel{H}{\longrightarrow}    ~~~  \mathcal{P}(\omega) & &  \\[2mm]
\downarrow \pi_ I            & ~~~~ ~~\downarrow \pi_ J          &&     \\[2mm]
   \mathcal{P}(\omega) / \ I      &  \stackrel{\Phi}{\longrightarrow}  ~~ \mathcal{P}(\omega) / J
\end{array}
$$

\hs  commutes. Note that any such  lifting $H$ preserves intersections modulo the range ideal $J.$ 
Note also that in general $H: \mathcal{P}(\om)\rightarrow \mathcal{P}(\om)$ does not need to be a 
Boolean algebra homomorphism.
It is therefore quite natural to ask for conditions on the given ideals $I$ and $J$ on $\om$ and the homomorphism 
$\Phi: \mathcal{P}(\om)/I\rightarrow \mathcal{P}(\om)/J$ that would
guarantee the existence of liftings $H: \mathcal{P}(\om)\rightarrow \mathcal{P}(\om)$ that preserve the Boolean 
algebra operations of the algebra $\mathcal{P}(\om)$, even the infinitary ones. Such liftings 
$H: \mathcal{P}(\om)\rightarrow \mathcal{P}(\om)$ are called \emph{completely additive liftings}. 
Note that such completely additive liftings are always given by maps $h:\om\rightarrow\om$ in such a 
way that $$H(\{n\})=h^{-1}(n)\mbox{ for all }n<\om.$$
It follows that $H(X)=h^{-1}(X)$ for all $X\subseteq \om$ and so from this we can conclude that every 
completely additive lifting $H: \mathcal{P}(\om)\rightarrow \mathcal{P}(\om)$ is a \emph{continuous map} 
when we make the natural identification of $\mathcal{P}(\om)$ with the Cantor set $2^\om$. Thus {\rm AT} 
asserts the seemingly weak form of this, the \emph{local continuity} of liftings between quotient algebras over 
${\bf \Sigma^1_1}$-ideals $I$ of subsets of $\om.$ While local continuity of liftings is a matter of additional axioms 
of set theory, the second author (see, for example, \cite{todorcevic98}, Problem 1) has posed a problem about the 
natural mathematical counterpart of this asking under which conditions continuous liftings can be turned into completely 
additive ones. In subsequent work of Farah \cite{Farah} and Kanovei-Reeken \cite{KR} this conjecture has been 
verified for a very wide class of ideals $I$ of subsets of $\om.$ We shall use the following particular result  from this work, which is a reformulation of 
 \cite[Theorem 2]{KR}, using the fact that a continuous homomorphism $H:\mathcal{P}(\om)\rightarrow\mathcal{P}(\om^\xi)$ is actually completely 
additive. Notice that below the boolean algebras  $\mathcal{P}(\om^\xi)/ I_{\om^\xi}$  are a direct generalization of the boolean algebras 
$\mathcal{P}(\om^n)/ I_{\om^n}$   and 
that we shall in fact only use in the sequel the case where the ordinal $\xi\geq 1$ is an integer. 

\begin{The}(see \cite[Theorem 2]{KR}) \label{kanree}
For every countable ordinal $\xi\geq 1,$ if a homomorphism $$\Phi: \mathcal{P}(\om)\rightarrow \mathcal{P}(\om^\xi)/ I_{\om^\xi}$$  has a continuous lifting
$H:\mathcal{P}(\om)\rightarrow\mathcal{P}(\om^\xi)$ then it
also has a completely additive lifting, or in other words, there is a map $h:\om^\xi\rightarrow \om$ such that $$\Phi(X)=[h^{-1}(X)]_{I_{\om^\xi}}\mbox{ for all } X\subseteq \om.$$
\end{The}

\section{The isomorphism relation}
 We had proved in \cite{Fin-Tod} that there exist two $\om$-tree automatic boolean algebras $\mathcal{B}$ and $\mathcal{B}'$ such that:
(1)  {\rm  (ZFC + CH)}
$\mathcal{B}$ and $\mathcal{B}'$ are isomorphic. (2) {\rm   (ZFC + OCA) } $\mathcal{B}$ and $\mathcal{B}'$ are not  isomorphic.
We are going to prove a similar result for the class of  $\om^n$-automatic structures, for any integer $n\geq 2$ using {\rm AT} in place of {\rm OCA}.

 We first recall the following folklore result (see, for example, \cite{Farah}).
\begin{The}\label{ch}
\noi
 {\rm  (ZFC + CH)}   ~~The  boolean algebras $\mathcal{B}_n$, $n\geq 1$,  are pairwise  isomorphic.
\end{The}

 Notice that this result is an immediate consequence of the simple fact that each of the Boolean algebras $\mathcal{B}_n$, 
$n\geq 1$,  is $\aleph_1$-saturated. 
Therefore assuming {\rm CH}, as  the boolean algebras $\mathcal{B}_n$, $n\geq 1$, are all of cardinality $\aleph_1,$ a well-known Cantor's back 
and forth argument will give us the isomorphisms. (The reader may find these notions in a textbook on Model Theory, like \cite{Poizat}). 

\hsn Note that if $1\leq m\leq n$ the equality $\om^p\om^m=\om^n$ for $p=n-m$ transfers easily to the existence of a map
$f: \om^n \rightarrow \om ^m$ with the property that for every subset $X\subseteq \om^m,$ $X\in I_{\om^m}$ if and only if $f^{-1}(X)\in I_{\om^n}.$
It follows that the corresponding  map  $X\mapsto f^{-1}(X)$ is a lifting of an isomorphic embedding $\Phi: \mathcal{P}(\om^m)/I_{\om^m}\rightarrow \mathcal{P}(\om^n)/I_{\om^n}$ and therefore we have the following fact.

\begin{Pro}{\rm  (ZFC)}
The algebra  $\mathcal{B}_m$ is isomorphic to a subalgebra of  $\mathcal{B}_n$ whenever $m$ and $n$ are positive integers such that $m\leq n.$
\end{Pro}

We shall see that such an isomorphic embedding is not always possible if we have the inequality $m>n.$ 
We shall use the following consequence of a well-known result of Rotman \cite{Ro} which one can prove by an easy induction on $n\geq 1.$

\begin{Lem}\label{rotman}
Suppose that  $n$ is an integer $\geq 1$ and that $\beta$ is some ordinal. 
Then for every mapping $f:\om^n\rightarrow \beta$ there is $Y\subseteq \om^n$ of order type $\om^n$  such that the image $f(Y)$ 
is a subset of $\beta$ of order type at most $\om^n.$
\end{Lem}

\begin{proof}
Consider firstly a mapping $f:\om\rightarrow \beta$ where $\beta$ is some ordinal. If the order-type of $f(\om)$ is strictly greater than $\om$ then there is  a subset $Z$ of $f(\om)$ which has 
order-type $\om$. But then $Y=f^{-1}(Z)$ is a subset of $\om$ which has also order-type $\om$ and $f(Y)=Z$ has order-type $\om$. Assume now that the result is proved for every 
integer $1\leq p<n$ and let $f:\om^n\rightarrow \beta$ be a mapping where $\beta$ is some ordinal. The ordinal $\om^n$ can be decomposed into $\om$ successive intervals $(I_k)_{k\geq 1}$ 
of length $\om^{n-1}$. We can now consider the restriction $f_k$ of $f$ to the interval $I_k$. By induction hypothesis, for each integer $k\geq 1$ 
 there is a subset $Y_k$ of $I_k$ which has order-type $\om^{n-1}$ and 
such that $f_k(Y_k)$ has order type at most $\om^{n-1}$. The set $Y=\bigcup_{1\leq k}Y_k  \subseteq \om^n$ has order type $\om^n$ and its  image $f(Y)$ 
is a subset of $\beta$ of order type at most $\om^n.$

\end{proof}

 We are now ready to state the following result.

\begin{The}\label{ind}
  {\rm   (ZFC + AT) }   The $\om^n$-automatic boolean algebras $\mathcal{B}_n$, $n\geq 1$,   are pairwise non  isomorphic and in fact $\mathcal{B}_m$
  is not isomorphic to a subalgebra of $\mathcal{B}_n$ whenever $m>n\geq 1.$
\end{The}

\begin{proof}
Suppose that for some $m>n\geq 1$ there is an isomorphic embedding $$\Phi: \mathcal{P}(\om^m)/I_{\om^m}\rightarrow \mathcal{P}(\om^n)/I_{\om^n}.$$ 

Choose a lifting $H: \mathcal{P}(\om^m)\rightarrow\mathcal{P}(\om^n)$ for the isomorphic embedding, i.e., 
a map for which the following diagram commutes
$$\begin{array}{ccccc}
   \mathcal{P}(\omega^m)        & \stackrel{H}{\longrightarrow}    ~~~  \mathcal{P}(\omega^n) & &  \\[2mm]
\downarrow \pi_ {I_{\om^m}}            & ~~~~ ~~\downarrow \pi_{I_{\om^n}}          &&     \\[2mm]
   \mathcal{P}(\omega^m) / \ I_{\om^m}      &  \stackrel{\Phi}{\longrightarrow}  ~~ \mathcal{P}(\omega^n) / {I_{\om^n}}
\end{array}
$$

It follows that  for every $X,Y\subseteq \om^m,$
\begin{enumerate}
\item[(1)] $X\vartriangle Y\in I_{\om^m}$ if and only if $H(X)\vartriangle H(Y)\in I_{\om^n},$
\item[(2)] $H(X)\in \Phi([X]_{I_{\om^m}}).$
\end{enumerate}

So, in particular, $H$ preserves intersections modulo $I_{\om^n}$. Using the argument appearing in Remark \ref{rem2} above, there is an uncountable
family  $\mathcal{A}$ of pairwise almost disjoint (i.e., $A\cap B\in {\rm Fin}$ for all $ A\neq B$ from $\mathcal{A}$) infinite 
subsets of $\om^m$  such that $A\not\in I_{\om^m}$ for all $A\in \mathcal{A}.$ By {\rm AT} there is
$A\in \mathcal{A}$ and a finite decomposition
$A=\bigcup_{i<k}A_i$ and for every $i<k$  a continuous function\footnote{Recall the identifications $\mathcal{P}(A_i)=2^{A_i}$ and $\mathcal{P}(\om^n)=2^{\om^n}$ which are giving us the topologies to which the continuity refers to.} $F_i: \mathcal{P}(A_i)\rightarrow\mathcal{P}(\om^n)$ forming a lifting of $\Phi$ when restricted to $\mathcal{P}(A_i)/I_{\om^m},$ or in other words a function for which the restricted diagram
$$\begin{array}{ccccc}
   \mathcal{P}(A_i)        & \stackrel{F_i}{\longrightarrow}    ~~~  \mathcal{P}(\omega^n) & &  \\[2mm]
\downarrow \pi_ {I_{\om^m}}            & ~~~~ ~~\downarrow \pi_{I_{\om^n}}          &&     \\[2mm]
   \mathcal{P}(A_i) / \ I_{\om^m}      &  \stackrel{\Phi}{\longrightarrow}  ~~ \mathcal{P}(\omega^n) / {I_{\om^n}}
\end{array}
$$
commutes. In particular, for  $X\subseteq A_i,$ we have that $F_i(X)\in I_{\om^n}$ if and only if $X\in I_{\om^m}.$
Since $A\not\in I_{\om^m}$ and the decomposition $A=\bigcup_{i<k}A_i$ is finite there is some $i<k$ such that $A_i\not\in I_{\om^m}$. 
Fix $i<k$ such that $A_i\not\in I_{\om^m}.$ Let $B=H(A_i).$ By Theorem \ref{kanree}, 
there is a function $f_i: B\rightarrow A_i$ which induces the completely additive lifting 
$X\mapsto f_i^{-1}(X)$ of the  restriction of $\Phi$ to $ \mathcal{P}(A_i) / \ I_{\om^m}$. 
Since $F_i$ is also a lifting of this isomorphic embedding, we have that  
$$F_i(X)\vartriangle f_i^{-1}(X)\in I_{\om^n}\mbox{ for every }X\subseteq A_i.$$ 
It follows in particular that for $X\subseteq A_i,$
$$X\in I_{\om^m} \mbox{ if and only if } f_i^{-1}(X)\in I_{\om^n}.$$
By Lemma \ref{rotman}, we can find a set $Y\subseteq B$ of order type $\om^n$ whose image $X=f_i(Y)$ is a subset of $A_i$ 
of order type at most $\om^n<\om^m.$ But then, we have a subset $X$ of $A_i$ of order type $<\om^m$ whose preimage 
$f_i^{-1}(X)$ has order type $\om^n$ as it contains the set $Y,$ i.e. $X\in I_{\om^m}$ and $f_i^{-1}(X) \notin  I_{\om^n}$, 
 a contradiction. This completes the proof.
\end{proof}

\hsn  We can now state the following consequence of the above theorems.

\begin{Cor}
The isomorphism relation for $\om^2$-automatic (respectively,  $\om^n$-automatic  for $n>2$)
 boolean algebras    (respectively,    partial orders) is not determined by the
axiomatic system  {\rm  ZFC}.
\end{Cor}

\begin{proof} The result for $\om^n$-automatic boolean algebras,  $n\geq 2$, follows directly from Theorems \ref{ch} and  \ref{ind} and the fact
that the boolean algebras $\mathcal{B}_1$ and $\mathcal{B}_2$ are $\om^2$-automatic, hence also $\om^n$-automatic  for $n>2$.
For partial orders, we consider the $\om^2$-automatic structures $(\mathcal{P}(\om) / {\rm Fin}, \subseteq_1)$
and $(\mathcal{P}(\om^2)/ I_{\om^2}, \subseteq_2)$.
These two structures are isomorphic if and only if the two boolean algebras $\mathcal{B}_1$ and $\mathcal{B}_2$ are isomorphic, see \cite[page 79]{Jech}.
Then the result for partial orders  follows from the case of boolean algebras.
\end{proof}

\hsn Reasoning as in \cite{Fin-Tod} for $\om$-tree-automatic structures, we can now get similar results  for other classes of $\om^n$-automatic structures.

\hsn  First  a boolean algebra
$(B, \cap, \cup, \neg, {\bf 0},  {\bf 1})$ can be seen as  a commutative ring with unit element    $(B, \Delta, \cap, {\bf 1})$,
where $\Delta$ is the symmetric difference operation.
The operations of union and complementation
can be defined from the symmetric difference and intersection operations.
Moreover two boolean algebras $(B, \cap, \cup, \neg, {\bf 0},  {\bf 1})$ and $(B', \cap, \cup, \neg, {\bf 0},  {\bf 1})$ are isomorphic if and only if
the rings $(B, \Delta, \cap, {\bf 1})$ and $(B', \Delta, \cap, {\bf 1})$ are isomorphic.
For each integer $n\geq 1$, we  denote $\mathcal{R}_n=(\mathcal{P}(\om^n) / I_{\om^n}, \Delta, \cap,  {\bf 1})$
 the commutative ring associated with the  boolean algebra
$\mathcal{B}_n$.

\begin{The}\label{ind2}
\noi
\begin{enumerate}
\ite  {\rm  (ZFC + CH)}   The $\om^n$-automatic  commutative rings  $\mathcal{R}_n$, $n\geq 1$,  are pairwise  isomorphic.
 \ite  {\rm   (ZFC + AT) } The $\om^n$-automatic    commutative rings      $\mathcal{R}_n$, $n\geq 1$,   are pairwise non  isomorphic.
\end{enumerate}
\end{The}

  Recall that   $M_k(R)$ is the set of square matrices with $k$ columns and $k$
rows and coefficients in a given ring $R$.  If $k\geq 2$ then the set $M_k(R)$, equipped with addition and multiplication of matrices, is a non commutative ring.
The ring $M_k(R)$ is first-order interpretable in the ring $R$; each matrix $M$ being represented by a unique $k^2$-tuple of elements of $R$,  the addition and
multiplication of matrices are first order definable in $R$.

 On the other hand,  for each integer $n\geq 1$, the class of  $\om^n$-automatic structures is closed under first order interpretations.
Thus if $R$ is an  $\om^n$-automatic ring   then the ring of matrices
$M_k(R)$ is also $\om^n$-automatic. It is well known that  two rings $R$ and $R'$ are isomorphic if and only if the rings $M_k(R)$ and
$M_k(R')$ are isomorphic, (this is proved  for instance in  \cite{Fin-Tod}).
We now denote, for each integer $n\geq 1$,  $\mathcal{M}_n=M_k(\mathcal{R}_n)$, where $k\geq 2$ is a fixed integer.
So we can state the following result.

\begin{The}\label{ind3}
\noi
\begin{enumerate}
\ite  {\rm  (ZFC + CH)}  ~~ The $\om^n$-automatic  non commutative rings  $\mathcal{M}_n$, $n\geq 1$,  are pairwise  isomorphic.
 \ite  {\rm   (ZFC + AT) } The $\om^n$-automatic  non    commutative rings      $\mathcal{M}_n$, $n\geq 1$,   are pairwise non  isomorphic.
\end{enumerate}
\end{The}

 Consider now  the unitriangular group $UT_k(R)$ for some integer $k\geq 3$ and $R$ a unitary ring.
A matrix $M\in M_k(R)$ is in the  group  $UT_k(R)$  if and only if it is an upper triangular matrix
which has only coefficients $1$ on the diagonal, where $1$ is the multiplicative unit of $R$. The group $UT_k(R)$ is also first order
interpretable in the ring $R$. (It is actually a classical example of nilpotent group of class $k-1$, for more details, see \cite{Belegradek,Fin-Tod}).

 We denote
$\mathcal{U}_{n,k}=UT_k(\mathcal{R}_n)$ for each $n \geq 1$.
The groups   $\mathcal{U}_{n,k}$ are first order interpretable in the ring $\mathcal{R}_n$. Thus the groups
 $\mathcal{U}_{n,k}$  are
$\om^n$-automatic. We can now state the following result.

\begin{The}\label{ind4}
\noi Let  $k\geq 3$ be an  integer.
\begin{enumerate}

\ite  {\rm  (ZFC + CH)}   The $\om^n$-automatic groups
$\mathcal{U}_{n,k}$, $n\geq 1$,  are pairwise isomorphic.

\ite
{\rm   (ZFC + AT) }   The $\om^n$-automatic groups
$\mathcal{U}_{n,k}$, $n\geq 1$,   are pairwise non isomorphic.

\end{enumerate}
\end{The}

\begin{proof} The result follows from Theorem \ref{ind2},
from the fact that if  $R$ and $S$ are two isomorphic commutative rings then $UT_k(R)$ and $UT_k(S)$ are also  isomorphic, and from a result of
Belegradek  who proved in \cite{Belegradek} that if  $UT_k(R)$ and $UT_k(S)$ are
isomorphic, for some integer $k\geq 3$ and some
commutative rings $R$ and $S$,  then the rings $R$ and $S$  are also  isomorphic.
\end{proof}

\hsn Then we can now state  the following result.

\begin{Cor}
The isomorphism relation for $\om^2$-automatic (respectively,   $\om^n$-automatic for $n>2$)    commutative rings (respectively,  non-commutative rings,
 groups,  nilpotent groups of class $p \geq 2$)
is not determined by the axiomatic system  {\rm  ZFC}.
\end{Cor}

 An $\om^n$-automatic presentation of a structure is given by a tuple of $\om^n$-automata
$(\mathcal{A}, \mathcal{A}_=,  (\mathcal{A}_i)_{1 \leq i \leq k})$.
The tuple of $\om^n$-automata can be coded by a finite sequence of symbols, hence by a unique integer $N$.
If $N$ is the code of the tuple of  $\om^n$-automata $(\mathcal{A}, \mathcal{A}_=,  (\mathcal{A}_i)_{1 \leq i \leq k})$
 we shall denote $\mathcal{S}_N$ the
$\om^n$-automatic structure $( L(\mathcal{A}),  (R_i)_{1 \leq i \leq k}) ) / E_\equiv$.

\hsn
The isomorphism  problem  for $\om^n$-automatic structures is:

$$\{ (p, m) \in \om^2  \mid \mathcal{S}_p \mbox{ is isomorphic to } \mathcal{S}_m \}.$$

\hsn  We can now  infer from above independence results  the following one.

\begin{The}
The  isomorphism problem for $\om^2$-automatic (respectively,   $\om^n$-automatic for $n>2$)
boolean algebras (respectively, rings, commutative rings,
non commutative rings, non commutative groups, nilpotent  groups of class $p\geq 2$)
is neither a $\Si_2^1$-set nor a  $\Pi_2^1$-set.
\end{The}

\begin{proof} We prove first the result for $\om^n$-automatic boolean
algebras.  By Theorem \ref{ind} we know that if   {\rm ZFC} is consistent then there is a model ${\bf  V}$  of
{\rm (ZFC + AT) } in which the two $\om^2$-automatic boolean algebras $\mathcal{B}_1$ and $\mathcal{B}_2$ are not isomorphic. But the
inner model ${\bf L }$  of constructible sets in  ${\bf  V}$  is a
model of {\rm  (ZFC + CH)} so in this model the two boolean
algebras $\mathcal{B}_1$ and $\mathcal{B}_2$ are  isomorphic by Theorem \ref{ch}.

 On the other hand, Schoenfield's Absoluteness Theorem implies that every $\Si_2^1$-set (respectively,  $\Pi_2^1$-set)
 is absolute for all inner models of {\rm  ZFC}, see \cite[page 490]{Jech}.

 In particular, if the isomorphism  problem  for  $\om^n$-automatic boolean algebras, $n\geq 2$,  was a $\Si_2^1$-set
(respectively,  a $\Pi_2^1$-set), then it could not be a different subset of $\om^2$ in the models  ${\bf  V}$  and   ${\bf L }$ considered above.
Thus the isomorphism  problem  for  $\om^2$-automatic (respectively,   $\om^n$-automatic for $n>2$)
  boolean algebras is neither a  $\Si_2^1$-set nor a $\Pi_2^1$-set.

 The other cases of   rings, commutative rings,
non commutative rings, non commutative groups, nilpotent  groups of class $p\geq 2$, follow in the same way from
Theorems \ref{ind2}, \ref{ind3}, and \ref{ind4}.
\end{proof}

\begin{Rem}
We had proved in \cite{Fin-Tod} that there exist two $\om$-tree-automatic atomless  boolean algebras which are isomorphic under {\rm  OCA} but not
under {\rm   CH}. But all the  $\om^n$-automatic atomless  boolean algebras $\mathcal{B}_n$, $n\geq 1$, we have considered in this paper are also
 $\om$-tree-automatic. Thus we have also in some sense improved our previous result by showing the following one.
\end{Rem}

\begin{The}\label{ba}
There exist infinitely many $\om$-tree-automatic atomless  boolean algebras $\mathcal{B}_n$, $n\geq 1$, which are pairwise isomorphic under
{\rm  CH} and pairwise non isomorphic under {\rm   AT}.
\end{The}

Notice that we have also a similar result for partial orders, rings, commutative rings,
non commutative rings, non commutative groups, nilpotent  groups of class $p\geq 2$.

\section{Concluding remarks}
Khoussainov, Nies, Rubin, and Stephan proved in \cite{KNRS} that 
the automatic infinite boolean algebras are the finite products $B_{fin-cof}^n$ of the boolean algebra
 $B_{fin-cof}$ of finite or cofinite subsets of the set of positive integers $\om$. 
An open problem is to characterize completely the $\om^n$-automatic (respectively, $\om$-tree-automatic) boolean algebras. 
A similar problem naturally arises for other classes of $\om^n$-automatic (respectively, $\om$-tree-automatic) structures, like  groups, 
rings, linear orders, and so on. 

\hs {\bf Acknowledgements.}
We thank  the anonymous referees for  useful comments 
on a preliminary version of this paper.

\end{document}